\documentclass[12pt]{article}
\usepackage{amsmath,amsthm,amsfonts,amssymb,mathrsfs,bm,graphicx,amscd,epsfig,psfrag,verbatim,hyperref}
\usepackage[usenames]{color}
\usepackage{fancyhdr}
\usepackage{bbold}

\parskip 	\bigskipamount

\newtheorem{theorem}{Theorem}[section]
\newtheorem{lemma}{Lemma}[section]

\newtheorem{proposition}{Proposition}[section]

\newtheorem{remark}{Remark}[section]

\newtheorem{definition}{Definition}[section]
\newtheorem{claim}{Claim}

\newtheorem{conjecture}{Conjecture}[section]
\def\a{{\alpha}}
\def\b{\beta}

\def\Z{\mathbb{Z}}

\def\R{\mathbb{R}}
\def\C{\mathbb{C}}
\def\E{\mathbb{E}}
\def\P{\mathbb{P}}
\def\T{\mathbb{T}}

\def\indf{\mathbb{1}}

\def\eps{\epsilon}
\def\d{\mathrm{d}}
\def\del{\delta}

\def\la{\lambda}
\def\L{\Lambda}
\def\a{\alpha}
\def\b{\beta}
\def\d{\mathrm{d}}

\def\S{\Sigma}
\def\s{\sigma}

\def\ol{\overline}

\def\min{\mathrm{min}}
\def\max{\mathrm{max}}
\def\D{\Delta}

\def\zz{\mathcal{Z}}

\def\cL{\mathcal{L}}
\def\M{\mathcal{M}}

\def\G{\mathcal{G}}

\def\calN{\mathcal{N}}
\def\H{\mathcal{H}}
\def\c{\complement}

\def\e{\mathcal{E}}
\def\moc{\M_1(\C)}
\def\mor{\M_1(\R)}
\def\mocs{\M_1^{sym}(\C)}
\def\F{\mathcal{F}}

\def\f{F}

\def\fD{\mathfrak{D}}
\def\cC{\mathcal{C}}
\def\cF{\mathcal{F}}
\def\cG{\mathcal{G}}
\def\cL{\mathcal{L}}

\def\vec#1{\underline{#1}}

\def\probmeas{\M_1(\C)}

\renewcommand{\l}[0]{\left }
\renewcommand{\r}[0]{\right}

\makeatletter
\renewcommand*{\@cite@ofmt}{\hbox}
\makeatother
\makeatletter
\def\blfootnote{\gdef\@thefnmark{}\@footnotetext}
\makeatother

\begin{document}


\title{Point processes, hole events, and large deviations: random complex zeros and Coulomb gases}
\author{
\begin{tabular}{c}
 {Subhro Ghosh}\\  National University of Singapore\\ subhrowork@gmail.com
\end{tabular}
\and
\begin{tabular}{c}
{Alon Nishry}\\ Tel Aviv University\\ alonish@tauex.tau.ac.il
\end{tabular}
}
\date{}

\maketitle

\blfootnote{\textup{2010} \textit{Mathematics Subject Classification}: Primary 60G55; Secondary 60F10. \newline \textup{Keywords : Point processes, particle systems, Coulomb gases, random matrices, random polynomials, hole probabilities, large deviations, empirical measures.}
}

\begin{abstract}
We consider particle systems (also known as point processes) on the line and in the plane, and are particularly interested in ``hole'' events, when there are no particles in a large disk (or some other domain). We survey the extensive work on hole probabilities and the related large deviation principles (LDP), which has been undertaken mostly in the last two decades. We mainly focus on the recent applications of LDP-inspired techniques to the study of hole probabilities, and the determination of the most likely configurations of particles that have large holes.

As an application of this approach, we illustrate how one can confirm some of the predictions of Jancovici, Lebowitz, and Manificat for large fluctuation in the number of points for the (two-dimensional) $\beta$-Ginibre ensembles. We also discuss some possible directions for future investigations.

\end{abstract}

\section{Introduction}

Random point configurations, also known as point processes, have been an object of key interest in the last few decades - both in probability theory and in the statistical physics literature. The most extensive results have been obtained in Euclidean spaces of dimensions 1 and 2, although higher dimensions and other geometries have also been studied. 

A point process $\Pi$, usually defined to live on a Polish space $\S$ equipped with a regular Borel measure $\mu$, is a probability distribution over the space of locally finite point configurations on $\S$. We recall here that a Polish space is a separable and completely metrizable topological space.  It is well known (\cite[Chap. 9]{DV}) that, under mild conditions, the statistical behaviour of a point process is described by its various \textsl{$k$-point intensities} ($k=1,2,\dots$), which are roughly the joint probability densities of having particles at $k$ specified locations in $\S$. For almost all interesting point processes, these $k$ point intensities are absolutely continuous  with respect to $\mu^{ \otimes k}$ (referred to as the \textsl{background measure}), and the resulting Radon Nikodym derivatives are known as the $k$ point intensity functions. Often, in Euclidean spaces or other homogeneous spaces, key point processes exhibit \textsl{invariance}, which is to say that the law of the process is invariant under the isometries of $\S$.

The most fundamental example of a point process is the Poisson point process. The Poisson point process, defined on the space $\S$ with respect to the background measure $\mu$, is the unique point process on $\S$ that exhibits statistical independence of its point configurations in disjoint domains, with the particle count in a domain $D \subset \S$ obeying a Poisson distribution with mean $\mu(D)$. This characterizing property of spatial independence makes many important statistical properties easy to compute, which is the reason behind the popularity of the Poisson process as a probabilistic model for many real-world systems (\cite[Chap. 2]{DV}). At the same time, it renders the Poisson process ineffective in modeling many natural phenomena, particularly those involving local repulsion, like electron systems. 

In this context, several \textsl{natural} models have emerged which embody non-trivial spatial correlation, including local repulsion, and at the same time being amenable to analysis. These models often have their origin in statistical physics, principal among them being Coulomb systems. The Coulomb system of size $n$ in dimension $d$ and inverse temperature $\beta$ is given by the joint distribution \[p(x_1,\dots,x_n)=Z_n^{-1}\exp\l( \frac12 \b \left[ \sum_{i \ne j} \rho(|x_i - x_j|) - n\sum_{i=1}^n V(x_i) \right] \r),\] where $\rho$ is the fundamental solution to the Laplacian in $d$ dimensions (in particular, the logarithm function in $2$ dimensions), and $V$ is an external field (or confining potential). It is also of considerable interest to consider a similar system in 1D, with $\rho$ the logarithm function; this model (or rather its infinite particle limit) being popularly known as the Dyson log gas.

In 1 and 2 dimensions, at inverse temperature $\b=2$, the Coulomb system with logarithmic interactions (a.k.a. \textsl{Dyson log gas} in 1D)  is known to be a \textsl{determinantal} point process, meaning that its correlation functions are given by certain determinants. When $V(x) = |x|^2$, these ensembles can also be described as the set of eigenvalues of certain random matrices - in 1 dimension it is the Gaussian Wigner matrix (GUE),  and in 2 dimensions the Ginibre ensemble (having independent standard complex Gaussian entries). Both of these ensembles have well defined weak limits which are determinantal point processes with infinitely many particles. In 1 dimension, the fundamental solution to the Laplacian is $f(x)=|x|$, and it is natural to consider a Coulomb system with this interaction potential. This system has been extensively studied by Aizenmann, Lebowitz, Martin, Yalcin and others (see, e.g., \cite{AM}, \cite{AGL}, \cite{Ku}, \cite{MY} for some of the delicate results on this model).

Another important two-dimensional model is the zero set of the standard Gaussian Entire Function (henceforth abbreviated as GEF). The GEF is given by the Gaussian Taylor series
\begin{equation} \label{GEF} \\\f(z)=\sum_{k=0}^{\infty} \frac{\xi_k}{\sqrt{k !}} z^k,
\end{equation}
where the $\xi_k$-s are independent standard complex Gaussians. We mention that the truncation of this series at degree $n$ is called the Weyl polynomial of degree $n$. The study of the GEF and the Weyl polynomials have their origin in statistical physics, where they have been investigated in the context of quantum chaotic dynamics (\cite{BBL}).

An object of key interest in the study of point processes is the ``hole'' event $\H_R$, entailing that a large disk (or interval, according to the dimension) of radius $R$ around the origin does not contain any particles. Of course, this is a \textsl{rare event}, and $\P[\H_R] \to 0$ as $R \to \infty$. The quantitative asymptotics of how this decay takes place throws light on the statistical structure of the point process, and has been studied in fine detail for many key processes. A closely related, but much less understood question, pertains to what causes such a large ``hole'' to appear. This involves understanding the typical configuration of particles outside the ``hole'', and until recently such results were available only for $\b=2$ Coulomb systems in 1 and 2 dimensions (\cite{JLM}). Very recently, progress has been made on this front for the GEF zeros process, as well as for holes of general shapes for the Ginibre ensemble. This is based on large deviation techniques, which brings us to the third key object in this paper, namely large deviation principles (abbreviated henceforth as LDPs).

Roughly speaking, a sequence of random variables $X_n$, defined on a common Polish space $\S$, is said to satisfy an LDP with rate $a_n \uparrow \infty$, and rate function $I: \S \to \R_+$ if, for any `nice' set $\F \subset \S$, we have
$$
\P[ X_n \in \F ] \approxeq  \exp\left( - a_n \inf_{x \in F} I(x) \right), \quad n \to \infty.
$$
In the above display, $\approxeq$ is understood in the sense that $\frac{\log \P[ X_n \in \F ]}{a_n} \to  -\inf_{x \in F} I(x)$ as $n \to \infty$. For us, the most interesting case is when the random variables $X_n$ are \emph{empricial measures} of the points (that is, discrete counting measures of the points, normalized to be probability measure), and $\S$ is a space of probability measures on $\R^d$. Large deviation principles of this type have been extensively studied for various random matrix models (see, e.g., \cite{JLM}, \cite{Gu-ldp}, \cite{BG}, \cite{BZ}) in the last two decades. More recently, large deviation principles have been understood for several random polynomial models (see, e.g., \cite{ZZ}, \cite{GZ}, \cite{Zel}, \cite{BuZ}).

In (\cite{GN,AR})  the main ingredient of the approach to the ``hole configuration'' is to consider the ``hole'' event as a ``rare'' event in the setting of the LDP for the relevant matrix or polynomial ensemble. This intuitively leads to the conclusion that the (limiting) intensity measure of the particles outside the hole must be the minimizer of the large deviation rate functional, under the constraint of the existence of the hole. This approach seems to be rather promising in investigating related problems for point processes.
We provide more details on this approach in Section \ref{beta-ans}, where we study the two-dimensional $\beta$-Ginibre ensembles (also known as jellium or the one-component plasma).

The main thrust of this work is on a certain set of ideas that tie in point processes, large deviations and the study of the hole event. Such focus naturally leaves out several important strands of work related to various combinations of these concepts. For instance, we mention the recent series of works studying various fine properties of the large deviation principle for Coulomb systems. In particular, these works establish rigorous connections of the LDP to the concept of \textsl{renormalized energy} (\cite{BS},\cite{SS-1}, \cite{SS-2}, \cite{LSZ}, \cite{LS1}, \cite{RS}, \cite{ArSZ}). Another direction of recent investigations involves the study of spatial \textsl{rigidity} structures that arise in several of these natural models (\cite{GP}, \cite{G-1}, \cite{G-2}, \cite{GK}, \cite{Bu}, \cite{BDQ}, \cite{BQ}, \cite{O}). Beyond that, there is the extensive research on universality in random matrix ensembles (see, e.g., \cite{TV}, \cite{Er}). We will not pursue these matters here.

\section{Large deviations for empirical measures}

A sequence of random variables $X_n$, defined on a common Polish space $\S$, is said to satisfy an LDP with rate $a_n \uparrow \infty$ and a convex lower semi-continuous rate function $I: \S \to \R_+$ if, for any Borel measurable set $\F \subset \S$, we have 
\[  \varlimsup_{n \to \infty} \frac{1}{a_n} \log \P[ X_n \in \F ] \le -\inf_{x \in \F^o} I(x),  \] where $\F^o$ is the interior of the set $F$ and 
\[  \varliminf_{n \to \infty} \frac{1}{a_n} \log \P[ X_n \in \F ] \ge -\inf_{x \in \ol{\F}} I(x), \] where $\ol{\F}$ is the topological closure of the set $\F$.

\begin{definition}
A rate function $I: \S \to \R_+$ is \emph{good} if all its level sets $\{x : I(x) \le \alpha \}$ are compact subsets of $\S$.
\end{definition}

\subsection{Eigenvalues of random matrices}

Large deviations for empirical measures of random matrices have been studied by multiple authors. In this section, we will only focus on LDPs for the empirical measure of some specific families of random matrices, including Gaussian (and other unitarily invariant) Hermitian ensembles in 1D, and the  (real and complex) Ginibre ensemble in 2D. We direct readers interested in a more extensive survey, including dynamical aspects related to evolution under the Dyson Brownian motion, to \cite{Gu-ldp}.

To describe the results, we need to introduce some notation. We will denote by $\M_1(\R)$ and $\M_1(\C)$, respectively, the space of probability measures on $\R$ and $\C$. For a finite set of points $\L:=\{z_1,\dots,z_N\}$, we define the empirical measure \[ \e(\L):= \frac{1}{N} \sum_{i=1}^N \del_{z_i},  \] where $\del_\la$ is the delta measure (of unit mass) at the point $\la$.  The empirical measures of eigenvalues live in the space $\M_1(\R)$ (or $\M_1(\C)$ as appropriate), and for ``nice enough'' ensembles, obey LDPs in the same space.

\subsubsection{The Ginibre ensemble}
We begin with the LDP for the Ginibre ensemble. For the original paper, we refer the reader to \cite{BZ}. The (real or complex) Ginibre ensemble (of order $n$) is the ensemble of eigenvalues of $n \times n$ random matrices with i.i.d. Gaussian entries (resp., real or complex) with mean zero and variance $n^{-1}$. The (infinite) Ginibre ensemble is the limit, in distribution, of the finite Ginibre ensembles.  

In the complex case, the (infinite) Ginibre ensemble turns out to be the 2D Coulomb gas at inverse temperature $\beta=2$. It also turns out to be a determinantal point process on $\C$ with kernel $K(z,w)=\exp(z \ol{w})$ and background measure $\d \mu(z)=\pi^{-1} e^{-|z|^2} \d m(z)$, i.e. the standard Gaussian measure. Its distribution is invariant under the rigid motions of the plane, and it serves as a crucial example of an (invariant) 2D point process that is relevant to the physical literature. The finite $n$ joint density, also known as the density of states in the physical literature, is given by
\begin{equation} \label{ginden} \varrho(z_1,\dots,z_n) =|\D(z_1,\dots,z_n)|^2 \exp\l(-\sum_{i=1}^n |z_i|^2\r), \end{equation}
where
$$ \D(z_1,\dots,z_n) = \prod_{j < k} (z_j - z_k) $$
denotes the Vandermonde determinant.

In what follows, let $S=\R$ or $\C$, with the particular value being specified by the context. Let $\mocs$ denote the space of probability measures on $\C$ that are symmetric about the real axis. Define the \emph{logathmic potential} of a measure $\mu \in \moc$
\[ U_\mu(z) = \int \log |z-w| \, \d \mu(w). \]
The \emph{logarithmic energy} of $\mu$ is given by
\[ \S(\mu) = \int U_\mu(z) \, \d \mu(z) = \iint \log |z-w| \, \d \mu(z) \d \mu(w). \] Note that $\S(\mu)$ can equal $-\infty$, e.g., in the case of measures $\mu$ having an atomic component.
Define also the functionals $I^{[S]}: \M_1(\C) \mapsto [0,\infty] $ as
\[I^{[\R]}(\mu) = \begin{cases}  \frac{1}{2} \int |z|^2 \d \mu(z) -  \frac{1}{2} \S(\mu) - 3/8  &\mbox{ if } \mu \in \mocs  \\ \infty &\mbox{ otherwise }   \end{cases} \]
and
\[I^{[\C]}(\mu) =  \int |z|^2 \d \mu(z) -  \S(\mu) - 3/4\]
We can state the LDP for the (real or complex) Ginibre ensemble (denoted resp. $\G_n^{[\R]}$ or $\G_n^{[\C]}$)  as follows:
\begin{theorem}(\cite{BZ}, \cite{HP1})
	The sequence of empirical measures $\e(\G_n^{[S]})$ obey a large deviation principle in the space $\M_1(\C)$ with rate $n^2$ and good rate function $I^{[S]}$.
\end{theorem}	
The rate function $I^{[S]}$ is minimized by the uniform measure on the unit disk.
Consequently, the (random) empirical measures $\e(\G_n^{[S]})$ converge a.s. to the uniform measure on the unit disk. 

\subsubsection{One-dimensional $\log$-gas}
We now consider $n$ particles on the real line, whose joint probability density is given by
\[Z_n^{-1} |\D(x_1,\dots,x_n)|^{\b} \exp\left(-n \sum_{i=1}^n V(x_i) \right), \]
where $\b>0$, the confining potential $V:\R \to \R$ is a continuous function such that, for some $\b'>1$ satisfying $\b' \ge \b$, we have 
\[ \varliminf_{|x| \to \infty} \frac{V(x)}{\b' \log |x|} >1, \]
and $Z_n = Z_n(\beta, V)$ is a normalizing constant.
Important special cases include $\b=1,2$ and $V(x)=|x|^2/2$, which correspond to the eigenvalues of the well-known Gaussian Orthogonal Ensemble (GOE) and Gaussian Unitary Ensemble (GUE) respectively. In other words, $\b=1$ and $2$ respectively correspond to the Hermitian versions of the real and complex Ginibre ensembles (tridiagonal random matrix models for general $\b > 0$ are known for the quadratic confining potential, see \cite{DE}).

To state the large deviations principle for such Hermitian random matrices (see, e.g., \cite{BG}, \cite{AGZ}), we introduce the rate function $I^V_\b : \mor \to [0,\infty]$ as:
\[I^V_\b = \begin{cases} \int V(x) \d \mu(x) - \frac{\b}{2} \S(\mu) -c^V_\b & \mbox{ if } \int V(x) \d \mu(x) < \infty \\ \infty &\mbox{ otherwise },  \end{cases} \]
where \[c^V_\b =  \inf_{\nu \in \mor} \bigg\{  \int V(x) \d \nu(x) - \frac{\b}{2} \S(\nu)  \bigg\}. \]
With these definitions in hand, we state the LDP for the above Hermitian random matrices as follows: 
\begin{theorem}(\cite{BG})
	The sequence of empirical measures $\e(\P^n_{V,\b})$ obey a large deviation principle in the space $\mor$ with rate $n^2$ and good rate function $I^V_\b$.
\end{theorem}
It is known that $I^V_\b$ attains its minimum value in the space $\mor$ at a unique measure $\s^V_\b$ that is compactly supported and is characterized by \[ V(x) - \b U_{\s^V_\b}(x) \begin{cases} = C^V_\b  &\mbox{ for } \s^V_\b-a.e. x \\ > C^V_\b &\mbox{ for all } x \notin \mathrm{supp}(\s^V_\b) \end{cases},\]
where $U_\mu$ denotes the logarithmic potential of the measure $\mu$, and $C^V_\b$ is some constant. An upshot of this is that the (random) empirical measures $\e(\P^n_{V,\b})$ converge a.s. to the measure $\s^V_\b$.

We mention in passing that large deviation principles are also known for  $\b$-Ginibre ensembles (defined in analogy to the $\b$ ensembles in 1D by using a general $\b$ exponent on the Vandermonde in \eqref{ginden}); these correspond to the 2D Coulomb gas (for general inverse temperature $\b$). For details, we refer to \cite{HP2}.

\subsection{Zeros of random polynomials}

The theory of large deviations for empirical measures of zeros of random polynomials is of more recent origin. One of the earliest articles in this 
 direction, namely (\cite{ZZ}), deals with the crucial case of (complex) Gaussian random polynomials, i.e., random polynomials with independent Gaussian coefficients (with mean zero and possibly decaying variances). Depending on the mode of decay of the variances, we obtain several distinguished ``standard ensembles'' - Kac (constant variance of coefficients), Elliptic (coefficient of $z^k$ has variance ${n \choose k} k!$) and Weyl (coefficient of $z^k$ has variance $1/k!$). \cite{ZZ} covers all these cases, as well as more general scalings of coefficients. In fact, \cite{ZZ} works in the more general setting of Gaussian measures on polynomial spaces of degree $n$ that live on the Riemann surface $\C\P^1$. These Gaussian measures are determined by inner products naturally induced from a metric $h$ and measure $\nu$ on $\C\P^1$, the only condition being that the pair $(h,\nu)$ satisfy the so-called Bernstein-Markov property. The results obtained on the Riemann sphere $\C\P^1$ can be transferred (via the stereographic projection) to the complex plane. For a detailed exposition of this, we refer the reader to \cite{But} (which also deals with the case of real Gaussian coefficients).

\subsubsection{Weyl polynomials}\label{weyl-polynomials}
In this article, we will focus on the crucial case of the Weyl polynomials,
$$
P_n(z) = \sum_{k=0}^{n} \frac{\xi_k}{\sqrt{k !}} z^k,
$$
which are naturally related to the so-called standard planar Gaussian Entire Function (GEF, see \eqref{GEF}).  Viewed over the complex plane, this corresponds to $h$ the standard Euclidean metric and $\nu$ the standard Gaussian measure on $\C$.

In what follows, we will denote by $\zz_n = \{z_1, \dots, z_n \}$ the zero set of the Weyl polynomial of degree $n$, \emph{scaled down by} $\sqrt{n}$. Also recall that for any measure $\mu \in \moc$, we denote by $U_\mu$ and $\S(\mu)$ the logarithmic potential and the logarithmic energy of $\mu$, respectively. We can now state the following LDP for zeros of Weyl polynomials:

\begin{theorem}(\cite{ZZ}, \cite{BuZ})
	The sequence of empirical measures $\e(\zz_n)$ satisfy a large deviation principle in the space $\moc$ with rate $n^2$ and good rate function 
	\[I^Z(\mu)= 2 \sup_{z \in \C} \l( U_\mu(z) - \frac12|z|^2 \r) - \S(\mu) -C. \]
\end{theorem}
The minimizer (and, consequently, the a.s. limit of the $\zz_n$-s ) of the above rate function is the uniform measure on the unit disk. The constant $C$ is such that $I$ evaluated at this measure is 0.

A word is in order here about the `unusual' form of the rate function, and especially the appearance of the non-linear and rather non-differentiable $\sup$ term. The key to this lies in an expression for the joint density of (the scaled) zeros for the Weyl polynomial (w.r.t. Lebesgue measure on $\C^n$), which can be written as 
\begin{equation}\label{weyl-joint-dens}
\rho(z_1,\dots,z_n) \propto |\D (z_1,\dots, z_n)|^2 \l( \frac{n}\pi \int_\C |Q_{\zz_n}(w)|^2 e^{-n |w|^2} \,\d m(w) \r)^{-(n+1)},
\end{equation}
where $Q_{\zz_n}$ is the monic polynomial
\[ Q_{\zz_n}(w) = (w-z_1)\dots (w-z_n), \]
and $m$ is Lebesgue measure on $\C$.
The Vandermonde determinant leads to the logarithmic energy term in the rate function. 
In addition, we have \[  |Q_{\zz_n}(w)|^2 e^{-n|w|^2} = \exp \l( 2n \left[U_{\e(\zz_n)}(w) - \frac12 |w|^2 \right] \r), \] and we see that for large $n$ the main contribution to the integral in \eqref{weyl-joint-dens} is coming from the maximum over $w \in \C$ of the term inside the square brackets.

\subsubsection{Other polynomials}

LDPs are known for empirical measures of zeros of many other random polynomial and polynomial-like ensembles, in addition to the models described above. Some key examples are \cite{FZ,Zel,GZ}, and \cite{BuZ}.

Some of these ensembles pose specific technical challenges of their own in establishing the LDP. As an example, we can consider the LDP for the Kac polynomial ensemble with exponential coefficients (\cite{GZ}). The major new difficulty is that all the coefficients are now positive a.s. This restricts the possible zero sets of such polynomials, the precise nature of which was not fully understood until recently. E.g., Obrechkoff's Theorem (\cite{Ob}) provides a necessary (but not sufficient) condition that the number of zeros of such a polynomial in a conical sector (around $\R_+$) can grow at most linearly with the angle at the apex of the cone. This issue makes an impact even on the form of the LDP rate functional:
\begin{theorem}(\cite{GZ})
	The empirical measure of zeros of Kac polynomials with exponential coefficients satisfy an LDP at rate $n^2$ and good rate function
	\[ I(\mu)= \begin{cases} U_\mu(1) - \frac{1}{2} \S(\mu) & \mbox{ if } \mu \in \mathcal{P} \\ \infty &\mbox{ otherwise}, \end{cases} \]
    where $\mathcal{P}$ is the set of all measures in $\moc$ that are weak limits of empirical measures of polynomials with positive coefficients.
\end{theorem}
The approach of \cite{GZ} exploits certain aspects of a potential theoretic description of the set $\mathcal{P}$ obtained by \cite{BE}. The universality results of \cite{BuZ} employ comparison techniques with appropriate ensembles already known to have an LDP. 

\section{Hole events and hole probabilities}
Hole events and hole probabilities have classically been a key object of interest in the study of point processes (a.k.a. particle systems). An important example of this is the well-known result that hole probabilities for a determinantal point process are given by certain Fredholm determinants related to its kernel (\cite[Chap. 6]{M}, \cite{ShT}).

To fix ideas, let $D_r$ denote the (open) disk (in dimension one an interval) of radius $r$, centered at the origin. The \emph{hole event}, denoted $\H_r$, is the event there are no points of the point process in $D_r$. The hole probability at radius $r$ is $\P[\H_r]$ which clearly decays to $0$ as $r \to \infty$. A very well-studied question in point process theory is the manner of decay of $\P[\H_r]$ (more precisely, its logarithmic asymptotics). Typically, the logarithm of the hole probability decays like a power law, whose exponent depends upon  the point process under consideration, and is thought to shed light on its `rigidity'. By rigidity in this setting, we envisage lattice-like behaviour. In particular, the heuristic is that faster the decay rate of the hole probability (that is, higher the exponent discussed above), stronger is the lattice-like behaviour. E.g., as we shall see below, the exponent for the Poisson process (in 2D) is 2. On the other hand, for compactly supported i.i.d. perturbations of the lattice $\Z^2$, the hole probability is 0 for large enough hole sizes, and hence, heuristically speaking, the above exponent is $\infty$.

In 2D, the simplest example of a homogeneous point process, namely the Poisson process (with unit intensity) gives a decay of $\P[\H_r]=\exp(-\mathrm{Area}(D_r))=\exp(-\pi r^2)$, so the decay exponent is 2. For the 2D Coulomb gas (inverse temperature $\b=2$, a.k.a. the Ginibre ensemble), it has been shown that this exponent is 4 (\cite{Sh1}, \cite[Chap. 7]{HKPV}), i.e. the hole probability exhibits the decay $r^{-4}\log \P[\H_r] \to -\frac14$ as $r \to \infty$. The larger exponent of the Ginibre process already attests to a stronger global spatial correlation compared to the Poisson process (the latter being characterized by the spatial independence of its points). For the application of LDP techniques to study hole probabilities for the Ginibre ensemble, see for example Section \ref{genhole}.

A key ingredient in the proof is the fact that the number of particles in $D_r$, for any determinantal point process is given by a sum of independent Bernoulli random variables (see, e.g., \cite[Chap. 4]{HKPV}). The parameters (success probabilities) of these Bernoullis are essentially the eigenvalues of the integral operator given by the kernel of the determinantal process restricted to $D_r$. An alternative approach for the Ginibre ensemble is to use the fact (first proved by Kostlan for the finite Ginibre ensemble) that the set of the squares of the moduli of the eigenvalues is distributed like a set of \emph{independent} Gamma random variables (see \cite[Theorem 4.7.3]{HKPV} also for the infinite ensemble).

In comparison, the study of hole probabilities for the zeros of the GEF introduces considerable challenges. The basic underlying reason for this is the absence of any tractable ``integrable'' structure in the GEF zeros process, as opposed to Poisson (spatial independence) or the Ginibre (determinantal). The study of the hole probability for the GEF has been undertaken in a series of papers, beginning with upper and lower bounds for $r^{-4}\log \P [\H_r]$ (\cite{STs-1}), and culminating in the proof of the fact (\cite{Ni1}) that 
\begin{equation}\label{gef-holeprob}
r^{-4}\log \P [\H_r] \to -\frac{e^2}{4}, \quad \mbox{as } r \to \infty. 
\end{equation} In fact, in \cite{Ni1} and subsequent works (\cite{Ni2}, \cite{Ni3}), hole probability asymptotics have been understood for a wide class of Gaussian entire functions. One can recover \eqref{gef-holeprob} using LDP techniques, and also study in more details the hole event, see Sections \ref{GEF-zero-cond-dist} and \ref{cond-dist-GEF}.
 
Thus, the exponents of the Ginibre and the GEF zeros process match. This leads to the interesting question regarding the comparison of the these two processes vis-a-vis their strength of correlations (or lattice-like behaviour). This has spawned an interesting collection of results. On one hand, there are comparison theorems for finite order correlation functions of the Ginibre and GEF zero ensembles (\cite{NS1}). On the other hand, there are recent results showing significant differences in the properties of their (spatially) conditional distributions (\cite{GP},\cite{GN}). A very interesting problem is to determine whether there is a natural invariant point process in the plane, whose hole probability decays qualitatively faster than decay rate of the Ginibre ensemble and the GEF zeros process.

We conclude this section by mentioning several other works related to hole probabilities, mostly of recent origin. An important instance is the study of hole probability asymptotics for zeros of a wide class of Gaussian analytic functions having a finite radius of convergence. This includes the well-studied hyperbolic GAFs (with general intensity $L>0$), whose domain of convergence is the unit disk. For the case $L=1$, the zero set has been shown to be a determinantal point process in \cite{PV}. In the same paper (\cite{PV}), the asymptotics of the hole probability (as $r \uparrow \infty$) has been worked out. In \cite{BNPS}, the hole probability asymptotics have been worked out for general $L$. In the process, a surprising discovery is made to the effect that the form of the asymptotics (including its dependence on $L$) depend crucially on whether $L$ is sub-critical ($0<L<1$), critical ($L=1$) or super-critical ($L>1$). Another interesting family of results involves gap probabilities (essentially, hole probabilities in 1D) for important families of 1D Gaussian processes, in particular connecting these asymptotics with simple properties of their spectral measures and so-called ``persistence probabilities'' (\cite{FeF},\cite{FeFS}, \cite{ABMO}, \cite{DM1}, \cite{DM2}). In \cite{Sh1} and \cite{Sh2} the author obtains fine quantitative estiamates on various aspects of the hole probability and the hole event for the Ginibre ensemble and related determinantal processes associated with higher Landau levels.

\section{Conditional distribution on the hole event}

In this section, we consider the following problem: what is the principal cause of a (rare) event of a hole of large radius? Having understood hole probabilities, the next natural question, therefore, is to try and understand the point process conditioned to have a hole of a large radius. This question, however, turns out to be a surprisingly difficult one - even in expectation. 

\subsection{The Ginibre ensemble}
Until recently, the only 2D point process for which this was understood was the Ginibre ensemble (\cite{JLM}; see \cite{Sh1} for a more recent study of finer aspects and more quantitative results). 
We state the result as (see the appendix in \cite{JLM}):
\begin{theorem}(\cite{JLM})
	\label{Ginibre-cond-intensity}
The conditional intensity $\rho^R$ (w.r.t. Lebesgue measure on $\C$) of the Ginibre eigenvalues on the event $\H_R$ is given by
\[\rho^R(z)= \frac1\pi e^{-|z|^2}\sum_{n=0}^{\infty} \frac{  |z|^{2n}}{\Gamma(n+1,R^2)}, \] for $|z| \ge R$, where we have used the incomplete Gamma function \[\Gamma(n+1,x)=\int_x^{\infty}e^{-t}t^n dt.\] 
\end{theorem}
In particular, for $r=R$ and $R\gg1$, we have $\rho(R)\sim \frac{1}{2}\pi R^2$ (see equation (2.13) in \cite{JLM}). This roughly corresponds to the appearance of a delta measure at the edge of the hole under appropriate renormalization. 

In \cite{Sh2}, Shirai described the complete behaviour of the conditional intensity of eigenvalues for the more general ``Ginibre-type'' ensembles. We mention here a version of Theorem 1.4 therein, adapted to our specific context and using our notations. For a Borel set $D\subset \C$,  let $\xi(D)$ denote the number of Ginibre eigenvalues in $D$. Let $A(x,y)$ denote the annulus $\{z \in \C : x \le |z| < y\}$.  With these notations, we can state:
\begin{theorem}[\cite{Sh2}]
 \label{Shirai}
 For $b>a>1$, we have \[\lim_{R \to \infty} \frac{1}{R^2}\E[\xi(A({\sqrt{a}R},{\sqrt{b}R}))|\H_R]=b-a\] and \[\lim_{R \to \infty} \frac{1}{R^2}\E[\xi(A(R,{\sqrt{b}R}))|\H_R]=b.\]
\end{theorem}
The discontinuity in the limit at $a=1$ captures the delta measure at the edge of the hole. The above limit aslo shows clearly that asymptotically, beyond the hole, the conditional intensity converges to the equilibrium intensity.

We point out that only the specific situation of a ``round'' hole  
was considered in this approach - that is, the hole consisted of no particles present in the disk of radius $R$, as opposed to, say a hole in the form of a particle-free square of side length $R$, with $R \to \infty$. This is crucial for obtaining the above results (as previously mentioned, the set of the squares of the moduli of the eigenvalues is distributed like a set of \emph{independent} Gamma random variables).

A crucial deficiency of this approach is the dependence on the above explicit description of the radii of the Ginibre points, which (or any substitute thereof) is not available for the other point processes. Even for the Ginibre ensemble, this approach depends crucially on the radial symmetry of the hole, and thus precludes any understanding of holes of any shape other than a disk.

\subsection{GEF zeros}\label{GEF-zero-cond-dist}
Very recently, progress has been achieved (\cite{AR},\cite{GN}) in understanding the conditional intensity around a large hole for point processes other than the Ginibre ensemble, and for non-circular holes. The new progress relies on a novel large deviation based approach. As an example, we state the following description for the (limiting) conditional intensity measure around a ``round'' hole for the GEF zeros process (\cite{GN}):

\begin{theorem}\label{conv-cond-meas-hole-GEF}
Let $m_\T$ be the uniform probability measure on the unit circle $\{|z|=1\}$, and put
\[  \mu^Z_0 =  e \cdot \d m_\T  + \indf_{\{|w| \ge \sqrt{e}\}} \frac{m }{\pi} ,\]
where 
Then, as $R \to \infty$, the scaled zero counting measure $[\zz_R]$ of the GEF, conditioned on having a hole in $\{|z| < R\}$, converges weakly to a limiting measure (in expectation, probability):
\[  \frac{1}{R^2} [\zz_R] \l( \frac{\cdot}{R} \r) \to  \mu^Z_0, \quad \mbox{as } R \to \infty. \]
\end{theorem}


We immediately point out a key difference with the conditional intensity for the Ginibre process: the appearance of a ``forbidden region'' immediately beyond the hole, where the expected density of zeros vanishes as $R \to \infty$. To the best of our knowledge, this is the first example of such a ``forbidden region'', and there is no instance, proven or conjectured, even in the physical literature that predicts such a phenomenon. See Figure \ref{GEF_zeros_hole} and Figure \ref{Ginib_hole} for a comparison between the distributions of the points conditioned on the hole event (with $R = 13$).

\begin{figure}
    \centering
    \begin{minipage}{0.45\textwidth}
        \centering
        \includegraphics[width=0.9\textwidth]{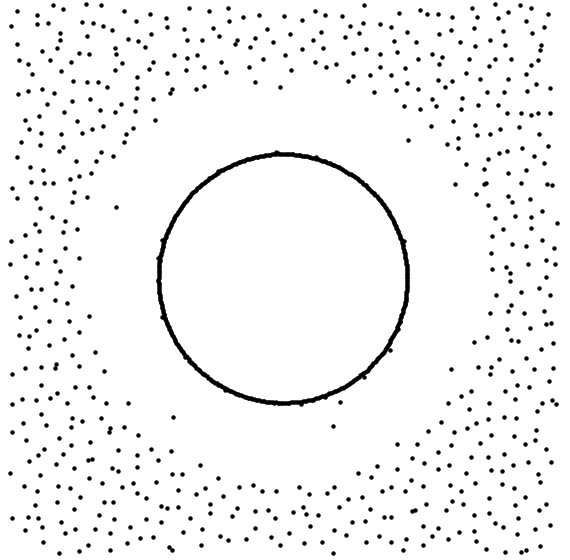} 
        \caption{GEF zeros, hole event}\label{GEF_zeros_hole}
    \end{minipage}\hfill
    \begin{minipage}{0.45\textwidth}
        \centering
        \includegraphics[width=0.9\textwidth]{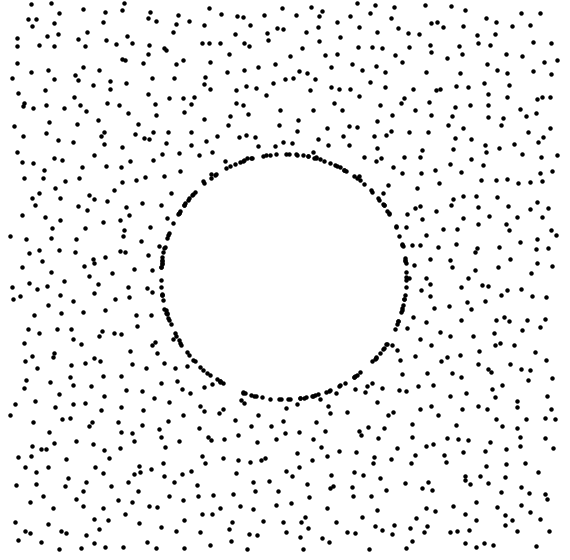} 
        \caption{Ginibre ensemble, hole event}\label{Ginib_hole}
    \end{minipage}
\end{figure}

The paper \cite{GN}, in fact, provides quantitative estimates on the typical number of zeros in the annulus between $R$ and $\sqrt{e} R$. In what follows, we denote by $N_{\f}(A)$ the number of zeros of the GEF in the set $A \subset \C$.
\begin{theorem}\label{num-zeros-in-annulus-GEF}
Suppose $R$ is sufficiently large, $\eps \in (R^{-2},1)$, that $\gamma \in (1+\frac{1}{2}\log \frac{1}{\eps} (\log R)^{-1} ,2]$, and consider the annulus \[A(R(1+\eps), \sqrt{e}R(1-\eps)) = \{ z \in \C : R(1+\eps) \le |z| \le \sqrt{e} R (1-\eps) \}. \]
We have  \[  \P\l[ N_{\f} (A(R(1+\eps), \sqrt{e}R(1-\eps))) \ge R^{\gamma}    | \H_R   \r]  \le \exp(-C \eps R^{2\gamma}),  \] where $C>0$ is a numerical constant.
\end{theorem}
The proofs of Theorem \ref{conv-cond-meas-hole-GEF} and Theorem \ref{num-zeros-in-annulus-GEF} are based on a certain deviations inequality for linear statistics of the GEF zeros. We provide more details in Section \ref{cond-dist-GEF}

It is an interesting problem to establish fine asymptotics, on the lines of Theorem \ref{Ginibre-cond-intensity}, for the GEF zeros process. This would involve, in the best case scenario, an explicit expression for the conditional density function. At a more modest level it can also envisage asymptotics of the conditional intensity function in various regimes, an important example of which is the rate of blowup of this function at the edge of the hole. It is also of interest to find the asymptotics of the (conditional) expected value of $N_F(A(R(1+\eps), \sqrt{e}R(1-\eps)) )$, as $R\to\infty$.


\subsection{Ginibre Ensemble: General holes and weighted Fekete points}\label{genhole}

Non-circular holes for the Ginibre ensemble were recently studied in the paper \cite{AR}. To this end, we recall the functional 
\[I(\mu)=\int |z|^2 \, \d \mu(z) - \iint \log |z-w| \, \d \mu(z) \d \mu(w) - \frac{3}{4}, \quad \mu \in \probmeas. \]
We mention that it is known that the uniform measure on the unit disk $\mu_0$ is the unique global minimizer for the above functional (and in fact, $I(\mu_0)=0$). We denote by $D$ the open unit disk. 
\begin{theorem}[\cite{AR}]
	\label{gincond}
	Let $\G_n$ denote the eigenvalues of the Ginbre emsemble of order $n$. Let $U \subset D$ be a subset satisfying at least one of the following conditions:
	\begin{itemize}
		\item (Balayage condition) There exists a sequence of open sets $U_n$ such that $\ol{U} \subset U_n \subseteq D$ for all $n$, and the balayage measure $\nu_n$ on $\partial U_n$ converges weakly to the balayage measure on $\partial U$.
		\item (Exterior ball condition) There exists $\eps >0$ such that for every $z \in \partial U$ there exists a $\eta \in U^\c$ such that 
		\[B(\eta, \eps) \subset  U^\c \text{ and } |z-\eta|=\eps. \]
	\end{itemize}
Then we have \[ \lim_{n \to \infty} \frac{1}{n^2} \log \P[\G_n(\sqrt{n} U)=0]=- \inf_{\mu \in \moc: \mu(U)=0} I(\mu). \]
\end{theorem}
Note that all convex domains satisfy the exterior ball condition.

In the case of the exterior ball condition, the proof relies on the use of \emph{weighted Fekete points}. Let $E \subset \C$ be a (nice) closed subset, and put
$$
\delta_n(E) = \sup_{z_1, \dots, z_n \in E} \left\{ \prod_{j<k} |z_j - z_k| \exp\left(-\frac12 |z_j|^2\right) \exp\left(-\frac12 |z_k|^2\right) \right\}^{\frac{2}{n(n-1)}}.
$$
A set $\cF_n = \{z_1^\star, \dots, z_n^\star\} \subset E$ is said to be an $n$-th weighted Fekete set if the points in $\cF_n$ attain the supremum $\delta_n(E)$ (such a set always exists, but is not necessarily unique). It is known (see \cite{SaT}) that the sequence $\{\delta_n(E)\}$ is decreasing, and furthermore
$$
\lim_{n\to \infty} \log \delta_n(E) = - \inf_{\mu(E^\c) = 0} \left[ \int |z|^2 \, \d \mu(z) - \iint \log |z-w| \, \d \mu(z) \right],
$$
where the infimum is over all probability measures $\mu$ such that $\mu(E^\c) = 0$. Heuristically, the Fekete points provide the most likely configuration of particles, conditioned on having no particles in the set $E^\c$.

For the infinite Ginibre ensemble, we have
\begin{theorem}{\cite{AR}}
	\label{hpgin}
	Let $\G_\infty$ denote the infinite Ginibre ensemble, and let $U \subset D$ be an open set satisfying either the balayage condition or the exterior ball condition as in the statement of Theorem \ref{gincond}. Then we have the hole probability asymptotics
	\[ \lim_{r \to \infty} \frac{1}{r^4} \log \P[\G_\infty (r U)=0] = - \inf_{\mu \in \moc: \mu(U)=0} I(\mu). \]
\end{theorem}


\section{Large fluctuations in the number of points and the Jancovici-Lebowitz-Manificat law}\label{largefluc}

Closely related to the hole event are the phenomena of ``deficiency'' and ``overcrowding'' in the number of particles, which entail that the number of particles in $D_r$ is very far from its typical value of about $r^2$ particles (with the standard normalization for the Ginibre ensemble and the GEF zeros). This has been extensively studied both for the Ginibre ensemble and the GEF zeros (\cite{JLM}, \cite{Kr1}, \cite{NSV1}), with the discovery that the fluctuations in both cases obey the Jancovici-Lebowitz-Manificat law (in short, the JLM law), that was first introduced in the context of large charge fluctuations for the 2D Coulomb gas (\cite{JLM}, see Conjecture \ref{JLMpred} for the statement).

We start with a special case of the JLM law. Denote by $n(R)$ the number of particles of the Ginibre ensemble inside the disk $\{|z| < R\}$. Here we consider the event $\{ n(R) = pR^2 \}$, where $p\ge 0$ ($p \ne 1$) is fixed. Shirai (\cite{Sh1}) proved

\begin{equation}\label{ShiraiLDP}
\lim_{R\to \infty} R^{-4} \log \P \left[ n(R) = pR^2 \right]= - \frac{1}{4} \left|2p^2 \log p - (p-1)(3p-1)\right|.
\end{equation}

In \cite{GN}, the authors derive the corresponding result for the GEF zeros, for a complete statement, we direct the reader to \cite{GN}.

\subsection{The Jancovici-Lebowitz-Manificat law}\label{JLMlawsubsec}
Compared with \eqref{ShiraiLDP}, one can certainly consider a wider range of fluctuations in the number of particles, and also examine other ensembles. We find it rather surprising, that the asymptotic decay of the probability of large fluctuations is described by a common `law', both for the Ginibre ensemble and the GEF zeros process (this law also appears in other ensembles, such as certain randomly perturbed lattices, see \cite{NSV1}).

\subsubsection{Finite $\beta$-Ginibre ensemble}

The physical paper \cite{JLM} by Jancovici, Lebowitz, and Manificat considers large charge fluctuations for a one-component Coulomb system of particles of one sign embedded into a uniform background of the opposite sign. This system is mathematically equivalent to the finite two-dimensional $\beta$-Ginibre ensemble, which consists of $N$ particles in the complex plane $\C$, whose joint probability density, with respect to Lebesgue measure on $\C^N$, is given by
\begin{equation}\label{betaGinibDensity}
p(z_1, \dots, z_N) = (Z_N^\beta) ^{-1} \prod_{j<k} \left|z_j - z_k \right|^\beta \exp\left( - \frac{\beta}{2} \sum_{j=1}^N |z_j|^2 \right). 
\end{equation}

Here $\beta > 0$ is the inverse-temperature, and $Z_N^\beta$ is the normalizing constant (also known as the partition function)

For $N$ large, the particles tend to be asymptotically uniformly distributed inside the disk of radius $\sqrt{N}$ centered at the origin. Let us denote by $n(R)$ the number of particles in the disk $D(0,R) = \left\{ |z| \le R \right\}$. For $N$ large compared with $R^2$ we have that $n(R)$ is typically about $R^2$.

Now, fix the parameters $a > \frac12$ and $b \ne 0$ (where $b \ge -1$ if $a = 2$, and $b > 0$ if $a > 2$), and consider the very rare event
$$ \calN^\beta_{a,b}(R) = \left\{ n(R) = \lfloor R^2 + b R^a \rfloor \right\}.$$
Based on macroscopic electrostatic considerations, the paper \cite{JLM} argues that after taking the limit $N$ to infinity, the following asymptotic probabilities for large fluctuations in $n(R)$ are observed.

\begin{conjecture}[The JLM law]\label{JLMpred}
With the parameters $a,b$ as above, and for $R \to \infty$, we have
$$\P \l[ \calN^\beta_{a,b}(R)\r] = \exp \l( -\beta \cdot \psi(\beta;a,b, \log R) \cdot R^{\varphi(a)} (1+o(1)) \r), $$
where
\[ \varphi(a) =
	\begin{cases}
		2a-1 &1/2 < a \le 1, \\
        3a -2 &1 \le a \le 2, \\
        2a &a \ge 2,
    \end{cases}
\]
\[
\psi(\beta;a,b, \log R) =
	\begin{cases}
		c_\beta b^2 &1/2 < a < 1, \\
		\frac16 |b|^3  &1 < a < 2, \\
		\frac12 (a - 2) b^2 \log R &a > 2,
	\end{cases}
\]
and $c_\beta$ is some constant depending on $\beta$.
\end{conjecture}
\begin{remark}
See \cite{JLM} for the precise expression for $\psi(\beta;a,b, \gamma)$ in the case $a = 2$ (cf. \eqref{ShiraiLDP}). It seems that no such expression is known for $a = 1$ (even when $\beta = 2$).

The constant $c_\beta$ is derived from the (conjectured) central limit theorem (CLT) for $n(R)$. Recently the CLT for smooth linear statistics was proved in \cite{LS2} (in this case the dependence on $\beta$ is explicit).
\end{remark}

In the case of the (infinite) Ginibre ensemble ($\beta = 2$) the arguments of \cite{JLM} are essentially mathematically rigorous (for proofs in the case $a = 2$, see the aforementioned \cite{Sh1}). For other values of $\beta$, it is not even known if a limiting object for the $\beta$-Ginibre ensemble exists, when the number of particles goes to infinity.

\subsubsection{Fluctuations for the GEF zeros process}

Nazarov, Sodin, and Volberg (\cite{NSV1}) confirmed that some of the predictions of \cite{JLM} hold also for large fluctuations in the number of zeros of the GEF. More precisely, they proved the following result.
\begin{theorem}(\cite{NSV1})
	\label{JLM}
	For every $a \ge 1/2$ and every $\eps >0$, we have 
	\begin{equation}
	\label{JLMeq}
	\exp(-R^{\varphi(a) + \eps}) \le \P\l( |n_F(R) - R^2 | \ge R^a \r)  \le 	\exp(-R^{\varphi(a) - \eps}), 
	\end{equation}
for all sufficiently large $R > R_0(\eps,a)$.
\end{theorem}

\begin{remark}
Partial results in this direction were already obtained in \cite{Kr1, STs-1}.
\end{remark}


Using the results of \cite{GN}, together with the approach of \cite{Kr1} it is possible to establish finer asymptotics for fluctuations in the GEF zeros process which are analogous to the JLM law, in a restricted range of exponents. As an example we mention

\begin{theorem}
	\label{JLM-ext}
	For fixed $b\ne 0,a\in (4/3,2)$, we have, for the GEF zeros process the asymptotics
	\[ \P \left[n_F(R) = \lfloor R^2 + b R^a \rfloor \right] = \exp\l( -\frac{2 |b|^3}{3}R^{3a-2} (1+o(1)) \r), R \to \infty. \]
\end{theorem}
The lower bound in the above asymptotics can, in fact, be shown to hold for $a \in (1,2)$, and it is plausible that the results hold in this range. 

\subsection{Deficiency and Overcrowding - conditional distribution}\label{defic-over}

Denote by $n_F(R)$ be the number of zeros of the GEF inside the disk $\{|z| < R\}$. Recall that the zero counting measure $[\zz_R]$ denotes the GEF zero counting measure, conditioned on the hole event in $\{n_F(R) = 0\}$. We now denote by $[\zz^p_R]$ the GEF zero counting measure, conditioned on the event $\{n_F(R) = \lfloor pR^2 \rfloor \}$, with $p\ge 0, p \ne 1$.

\textbf{Notation:}
If $p = 0$, we set $q = e$. Otherwise, for $0 < p < e$, let $q = q(p)$ be the non-trivial solution of the equation $p(\log p - 1) = q(\log q - 1)$.

In \cite{GN}, we find the limiting conditional measure for these conditional counting measures. More precisely, the scaled conditional counting measure converges weakly (say in expectation) to a limiting Radon measure on $\C$
\[  \frac{1}{R^2} [\zz^p_R] \l( \frac{\cdot}{R} \r) \to \mu^Z_p, \mbox{ as } \quad R \to \infty, \]
where
$$\mu^Z_p = \begin{cases}
\left[\indf_{\left\{ 0\le\left|w\right|\le\sqrt{p}\right\} } + \indf_{\left\{ \sqrt{q}\le\left|w\right|\right\} }\right]\frac{ m}{\pi}+\left(q-p\right) m_\T & \,p\in\left[0,1\right);
\vspace{0.1cm}
\\
\left[\indf_{\left\{ 0\le\left|w\right|\le\sqrt{q}\right\} } + \indf_{\left\{ \sqrt{p}\le\left|w\right|\right\} }\right]\frac{ m}{\pi}+\left(p-q\right) m_\T & \,p\in\left(1,e\right);
\vspace{0.1cm}
\\

\indf_{\left\{ \sqrt{p}\le\left|w\right|\right\} }\frac{m}{\pi}+p\, m_\T & \,p\ge e.
\end{cases}
$$
Using the determinantal structure of the Ginibre ensemble, it is not difficult to prove that a similar result holds for the (conditional) ``eigenvalue'' counting measure, with the limiting measure $\mu^Z_p$, replaced by
$$
\mu^G_p = \begin{cases}
\left[\indf_{\left\{ 0\le\left|w\right|\le\sqrt{p}\right\} } + \indf_{\left\{ 1\le\left|w\right|\right\} }\right]\frac{m}{\pi} +\left(1 - p\right) m_\T & \,p\in\left[0,1\right);
\vspace{0.1cm}
\\

\left[\indf_{\left\{ \left|w\right|\le 1 \right\} } + \indf_{\left\{ \sqrt{p}\le\left|w\right|\right\} }\right]\frac{m}{\pi}+\left(p - 1\right) m_\T & \,p > 1.
\end{cases}
$$

\begin{figure}
    \centering
    \begin{minipage}{0.45\textwidth}
        \centering
        \includegraphics[width=0.9\textwidth]{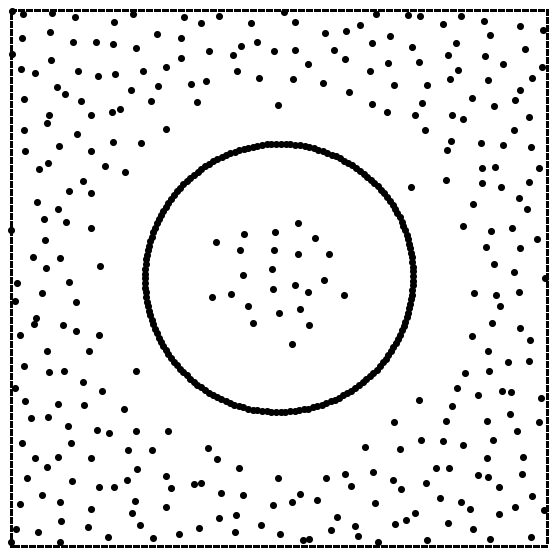} 
        \caption{GEF, deficiency $p = \frac12$}\label{GEF-defic}
    \end{minipage}\hfill
    \begin{minipage}{0.45\textwidth}
        \centering
        \includegraphics[width=0.9\textwidth]{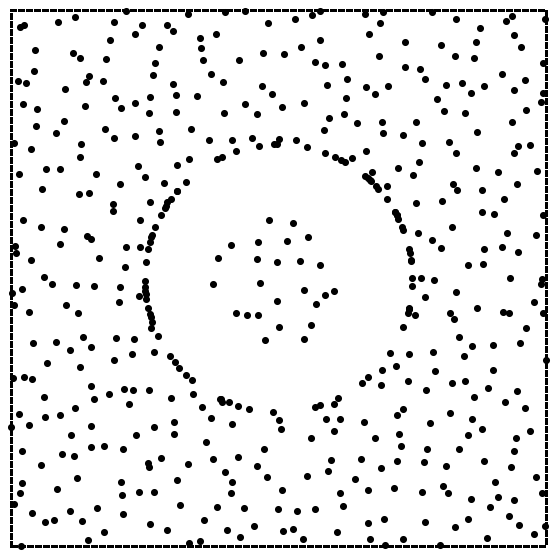} 
        \caption{Ginibre, deficiency $p = \frac12$}\label{Ginib-defic}
    \end{minipage}
\end{figure}

Figure \ref{GEF-defic} and Figure \ref{Ginib-defic} illustrate the case of a deficiency $p = \frac12$, while Figure \ref{GEF-over} and Figure \ref{Ginib-over} illustrate overcrowding for $p = 2$.

\begin{figure}
    \centering
    \begin{minipage}{0.45\textwidth}
        \centering
        \includegraphics[width=0.9\textwidth]{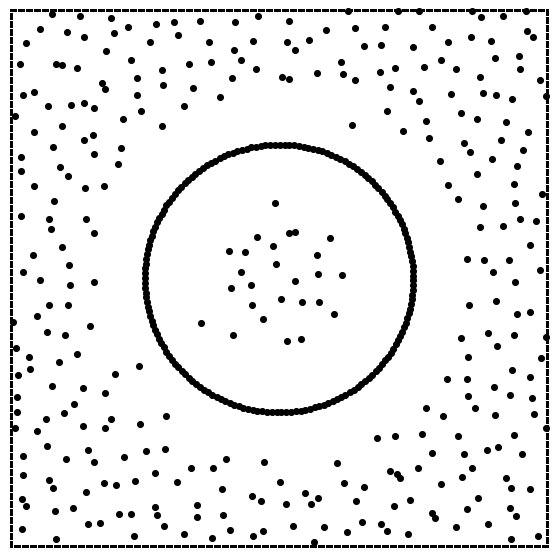} 
        \caption{GEF, overcrowding $p=2$}\label{GEF-over}
    \end{minipage}\hfill
    \begin{minipage}{0.45\textwidth}
        \centering
        \includegraphics[width=0.9\textwidth]{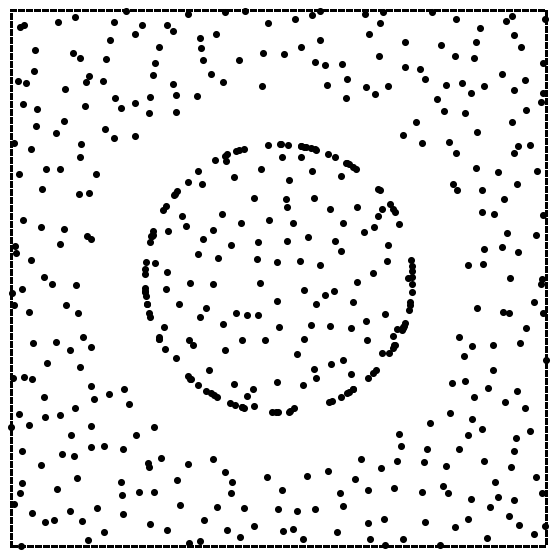} 
        \caption{Ginibre, overcrowding $p = 2$}\label{Ginib-over}
    \end{minipage}
\end{figure}

\newpage{}

\section{Analysis of an illustrative model: 2D $\beta$-ensembles}\label{beta-ans}

In this section, we will provide proof sketches for the convergence of the conditional distributions of the particles, and of the JLM law for the finite $\beta$-Ginibre ensembles.
In the next section we will briefly describe our proof from \cite{GN} for the zeros of the GEF. We recall that the two-dimensional $\beta$-Ginibre ensemble, consists of $N$ particles, whose joint probability density, with respect to Lebesgue measure on $\C^N$, is given by \eqref{betaGinibDensity}.

Since, at the moment, a limiting object for the finite $\beta$-Ginibre ensemble is not known to exist, we choose a different limiting procedure than the one in the paper \cite{JLM} (see Section \ref{JLMlawsubsec}). We fix a scaling parameter $\alpha\ge1$, and consider the asymptotics in terms of the large parameter $R = \sqrt{\frac{N}{\alpha}}$. Heuristically, with the parameter $\alpha$, the finite system resembles the (hypothetical) infinite system up to distances less than $\sqrt{\alpha} R$ from the origin.  In this setting, there are typically about $R^2$ particles in the disk $D(0,R) = \left\{ |z| \le R \right\}$. We again denote the number of particles in this disk by $n(R)$.

We will illustrate below the proof of some of the predictions in Conjecture \ref{JLMpred} above (using the different scaling procedure above). The proof in the case of overcrowding is similar. For $a\ge2$ one has to choose the value of the scaling parameter $\alpha$ depending on $a, b$, and also $R$ (if $a > 2$).

\begin{theorem} \label{JLM-beta-ens}
For fixed $b >0 ,a\in (4/3,2)$, we have,
\[
\P \left[n(R) \le R^2 - b R^a \right] = \exp\l( -\frac{\beta b^3}{3} R^{3a-2} (1+o(1)) \r), R \to \infty.
\]
\end{theorem}

Our proof of Theorem \ref{JLM-beta-ens} proceeds via large deviation type estimates, and for $a\le 4/3$, the error in our estimates (see Proposition \ref{dev-ineq}) overwhelms the leading term.
In the range $a \in (1,\frac43]$ establishing the JLM law for the $\beta$-Ginibre ensembles, with $\beta \ne 2$ is an interesting open problem.

In addition to predicting the asymptotic decay of the rare events, the paper \cite{JLM} describes the limiting conditional distribution of the particles. In order to simplify the presentation, we will consider just the events
$$ F_p(R) = \left\{ n(R) \le p R^2 \right\}, $$
with $p \in [0,1)$ a fixed parameter (one can also consider overcrowding, and $p$ can depend on $R$).

In order to describe the limiting distribution, we introduce \emph{linear statistics}, that is, the random variables
$$ n(\varphi ; R) = \sum_{j=1}^N \varphi\left(\frac{z_j}{R}\right), $$
where $\varphi$ is a smooth (say $C^2$) test function with compact support.

The following result describes the conditional limiting distribution, where $m$ is Lebesgue measure on $\C$, and $m_\T$ is the uniform probability measure on the unit circle $\{|z| = 1\}$.

\begin{theorem}
	\label{JLM-cond-dist}
	As $R \to \infty$ we have

$$ \E_{F_p(R)} \left[ n(\varphi ; R ) \right] = R^2 \int_\C \varphi(w) \,\d \mu_p^\alpha(w) + o(R^2),$$

where the limiting conditional measure is given by
$$ \mu_p^\alpha = \left(\indf_{\{|z| \le \sqrt{p}\}}  + \indf_{\{1 \le |z| \le \sqrt{\alpha}\}} \right) \frac{m}{\pi} + (1 - p) m_\T. $$
\end{theorem}

\begin{remark}
We write $\P_F$ (resp. $\E_F$) for the conditional probability (resp. expectation) on the event $F$.
\end{remark}

\begin{remark}
Notice $\mu_p^\alpha$ is \emph{not} a probability measure. Later, it will also be convenient to work with the normalized probability measure
$$ \overline{\mu}_p^\alpha = \frac1\alpha \mu_p^\alpha. $$
\end{remark}

\subsection{Deviation inequality for linear statistics}\label{dev-ineq-ginibre}

Theorem \ref{JLM-cond-dist} follows from the following deviation inequality:
\begin{proposition}\label{dev-ineq}
For $R, \lambda > 0$ we have
$$ \P_{F_p(R)} \left[\left|n(\varphi ; R) - R^2 \int_\C \varphi(w) \, \d \mu_p^\alpha (w) \right| \ge \lambda \right] \le 
\exp\left(-\frac{C \beta}{\fD (\varphi)} \lambda^2 + C_\varphi R^2 \log R\right),
$$
where $$\fD(\varphi) = \left\Vert \nabla \varphi \right\Vert^2_{L^2(m)} = \int_\C (\varphi_x^2 + \varphi_y^2) \, \d m.
$$
\end{proposition}

\begin{remark}
The proposition is non-trivial only for $\lambda \ge C R \sqrt{\log R}$, hence we may assume this holds below.
\end{remark}

The proof (motivated by the LDP approach in \cite{BZ, HP1}) is based on the approximation of the joint density of the particles (at the exponential scale) by a functional acting on probability measures. The (strictly convex, lower semi-continuous) function is given by
\begin{equation}\label{limiting_functional}
I_\alpha(\nu) = \int_\C \frac{|z|^2}{\alpha} \, \d \nu(z) - \Sigma(\nu).
\end{equation}
The global minimizer of the functional (also known as the \emph{equilibrium measure}) is given by
$$ \overline{\mu}^\alpha_{\mathrm{eq}} = \frac1\alpha \indf_{\{|z| \le \sqrt{\alpha}\}} (w) \cdot \frac{m}{\pi},$$
that is, the uniform probability measure on the disk $D(0, \sqrt{\alpha})$.

Consider now the set of measures,
$$ \cF_p = \cF_p(\a) = \left\{ \mu \in \probmeas : \mu(D) \le \frac{p}{\alpha}  \right\}, \quad p\in [0,1), $$
where $\probmeas$ is the set of probability measure on the complex plane $\C$, and $D$ is the unit disk $D = D(0,1) = \{ |z| < 1 \}$. In Section \ref{sec_min_meas} we show that the measure that minimizes the functional $I_\alpha$ over the closed set $\cF_p$ is $\overline{\mu}_p^\alpha$.

We also consider measures which are `far' from the minimizing measure $\overline{\mu}_p^\alpha$. For a test function $\varphi$ we put
$$ \cL_{\varphi, \lambda} = \left\{ \mu \in \probmeas : \left| \int_\C \varphi(w) \, \d \mu (w) - \int_\C \varphi(w) \, \d \overline{\mu}_p^\alpha (w) \right| \ge \lambda \right\}. $$

A key tool required for the proof of Proposition \ref{dev-ineq} is the next claim, which can be regarded as an effective formulation of the fact $I_\alpha$ is strictly convex.

\begin{claim}\label{convexity-of-functional}
For any compactly supported measure $\nu \in \cF_p \cap \cL_{\varphi, \lambda}$, we have

$$ I_\alpha(\nu) \ge I_\alpha(\overline{\mu}_p^\alpha (w))  + \frac{2\pi}{\fD(\varphi)} \lambda^2. $$
\end{claim}
\begin{proof}
See the similar \cite[Claim 11]{GN}.
\end{proof}

\subsection{Approximation of the joint density}\label{approx-joint-intens}

We start with an asymptotic estimate for the normalizing constant $Z_N^\beta$ (see \cite[Corollary 1.5]{LS1}).

\begin{proposition}\label{beta_ans_norm_const}
We have
$$ \log Z_N^\beta = -\frac12 \beta N^2 \cdot I_\alpha(\overline{\mu}^\alpha_{\mathrm{eq}}) + O(N \log N),$$
where the error term depends on $\beta$.
\end{proposition}

It is technically convenient to restrict the consideration to particles in a finite box. For example, one can easily prove that
$$
\P \left[\max_{j\in\{1,\dots,N\}}|z_j| \ge R^3  \right] \le \exp\left(-\frac{\beta}{4} R^6\right), \quad \forall R \ge R_0(\beta).
$$
Hereafter, we assume that $\max_j|z_j| \le R^3$.
\subsubsection{Smoothed empirical measure}
Given $\vec{w} = (w_1, \dots, w_N) \in \C^N$, let $\mu_{\vec{w}} = \frac1N \sum_{j=1}^N \delta_{w_j}$ be the empirical probability measure of the points. For a parameter $t = t(R)$ (which is chosen to be $R^{-C_1}$ for sufficiently large constant $C_1>0$), we consider the smoothed empirical measure
$$
\mu^t_{\vec{w}} = \mu_{\vec{w}} \star m_{|z|=t},
$$
where $m_{|z|=t}$ is the uniform probability measure on $|z|=t$. It is not difficult to verify the following (cf. \cite[Claim 4]{GN}).

\begin{claim}\label{approx_dens}
If $C_1$ is chosen to be sufficiently large, then
$$
\frac{1}{N^2} \sum_{j\ne k} \log|w_j - w_k| \le \Sigma(\mu^t_{\vec{w}}) + \frac{C \log \frac1t}{N} \le \Sigma(\mu^t_{\vec{w}}) + \frac{C \log R}{N},
$$
and
$$
\frac{1}{N} \sum_{j=1}^N |w_j|^2 = \int_\C |w|^2 \, \d \mu^t_{\vec{w}} (w) + O(R^{-2}).
$$

\end{claim}

Instead of working with the original particles $z_j$ it is more convenient to work with a scaled version. We set
$$
L = R - t, \quad \mbox{and } \vec{w} = \frac1L \vec{z}.
$$
With this scaling, our assumptions that $n(R) \le p R^2$ and $\left| n(\varphi; R) - R^2 \int_\C \varphi(w) \, \d \mu_p^\alpha(w) \right| \ge \lambda$ imply that
$$ \mu^t_{\vec{w}}(D) \le \frac{p}{\alpha}, $$
and
$$ \left| \int_\C \varphi(z) \, \d \mu^t_{\vec{w}}(z) - \int_\C \varphi(w) \, \d \overline{\mu}_p^\alpha (w) \right| \ge \frac{\lambda}{2N}, $$
where we used the fact that $N = \alpha R^2$, the (H{\"o}lder) continuity of $\varphi$, and our assumption that $\lambda \ge CR \sqrt{\log R}$.

\subsubsection{Upper bound for the joint density}
Define the following sets in $\C^N$:
\begin{eqnarray*}
N_p = N_p(t) & = & \left\{ \vec{w} \in \C^N : \mu^t_{\vec{w}}(D) \le \frac{p}{\alpha} \right\},\\
L_{\varphi,\lambda} = L_{\varphi,\lambda}(t) & = & \left\{ \vec{w} \in \C^N : \left| \int_\C \varphi(z) \, \d \mu^t_{\vec{w}}(z) - \int_\C \varphi(w) \, \d \overline{\mu}_p^\alpha (w) \right| \ge \frac{\lambda}{2N} \right\}.
\end{eqnarray*}
By the arguments outlined above, we have
\begin{align*}
\P \left[F_p(R) \cap \left\{\left| n(\varphi; R) - R^2 \int_\C \varphi(w) \, \d \mu_p^\alpha(w) \right| \ge \lambda \right\} \right]\quad\quad\quad\quad\quad\quad\quad\quad\quad\quad\\
\le \exp\left(\frac12 \beta N^2 \left[I_\alpha(\overline{\mu}^\alpha_{\mathrm{eq}}) - \inf_{\vec{w} \in N_p\cap L_{\varphi,\lambda}} I^\star (\vec{w})  \right] + O(N \log N) \right),
\end{align*}
where
$$
I^\star(\vec{w}) = \frac{L^2}{N^2} \sum_{j=1}^N |w_j|^2 - \frac{1}{N^2} \sum_{j\ne k} \log|w_j - w_k|.
$$

Claim \ref{approx_dens} gives the following bound
$$ I^\star(\vec{w}) \ge I_{\alpha^\prime}(\mu^t_{\vec{w}}) - \frac{C\log R}{N}, $$
where $I_\alpha$ is the limiting functional defined in \eqref{limiting_functional}, and $\alpha^\prime = \frac{N}{L^2} = \alpha (1 + O(R^{-2}))$. As one can show that
$$
\inf_{\mu\in \cF_p \cap \cL_{\varphi,\lambda}} I_{\alpha^\prime} (\mu) = \inf_{\mu\in \cF_p \cap \cL_{\varphi,\lambda}} I_{\alpha}(\mu) + O(R^{-2}),
$$
we conclude that
\begin{align}
\P \left[F_p(R) \cap \left\{\left| n(\varphi; R) - R^2 \int_\C \varphi(w) \, \d \mu_p^\alpha(w) \right| \ge \lambda \right\} \right]\quad\quad\quad\quad\quad\quad\quad\quad\quad\quad \nonumber \\
\le \exp\left(- \frac12 \beta N^2 \inf_{\mu\in \cF_p \cap \cL_{\varphi,\widetilde{\lambda}}} \left[I_{\alpha} (\mu) - I_\alpha(\overline{\mu}^\alpha_{\mathrm{eq}}) \right] + O(N \log N) \right), \label{prob-lin-stats-upper-bnd}
\end{align}
where $\widetilde{\lambda} = \frac{\lambda}{2 N}$. 

\subsection{Proof of Proposition \ref{dev-ineq}}

Using Claim \ref{convexity-of-functional} we obtain
\begin{equation}\label{funct-low-bnd}
\inf_{\mu\in \cF_p \cap \cL_{\varphi,\widetilde{\lambda}}} \left[I_{\alpha} (\mu) - I_\alpha(\overline{\mu}_p^\alpha) \right] \ge \frac{2\pi}{\fD(\varphi)} \cdot \left(\frac{\lambda}{2N}\right)^2
\ge \frac{C \lambda^2}{\fD(\varphi) N^2}.
\end{equation}
In Section \ref{lower-bnd-JLM} we show the lower bound estimate
\begin{equation}\label{low-bnd-prob-zero-defic}
\P\left[F_p(R)\right] \ge \exp\left(-\frac12 \beta N^2 \cdot \inf_{\mu \in \cF_p} \left[ I_\alpha(\mu) - I_\alpha(\overline{\mu}^\alpha_{\mathrm{eq}}) \right] + O(N \log N) \right).
\end{equation}
To prove Proposition \ref{dev-ineq} we combine \eqref{prob-lin-stats-upper-bnd}, \eqref{funct-low-bnd}, \eqref{low-bnd-prob-zero-defic}, and the fact that $\overline{\mu}_p^\alpha$ minimizes $I_\a$ over $\cF_p$.

\subsection{Outline of the proof of Theorem \ref{JLM-beta-ens}}
The proof of the upper bound of Theorem \ref{JLM-beta-ens} follows the same lines as the proof of Proposition \ref{dev-ineq} (ignoring the condition on linear statistics).
From Claim \ref{func-value-at-mu-p} we find (cf. \cite{Sh1})
\begin{eqnarray*}
\P \left[F_p(R)\right] & \le &
\exp \left( -\frac12 \beta N^2 \cdot \frac{1}{4\alpha^2} \left|2p^2 \log p - (p-1)(3p-1)\right| + O(N \log N)\right) \\
& \le & \exp \left( -\frac18 \beta R^4 \left|2p^2 \log p - (p-1)(3p-1)\right| + O(R^2 \log R)\right).
\end{eqnarray*}
To obtain the upper bound we take $p = 1 - b R^{a-2}$ (because of the error term the result is not trivial for $\frac43 < a < 2$).
\begin{remark}
Notice that for $p$ near $1$ we have
$$2p^2 \log p - (p-1)(3p-1) = \frac23 (p-1)^3 (1 + O(p-1)).$$
\end{remark}

\subsubsection{Lower bound}\label{lower-bnd-JLM}
The proof of the lower bound is similar to the proof of the lower bound for the hole probability in \cite{AR}, and we will only sketch the idea.
We again choose $t = R^{-C_1}$ for some sufficiently large $C_1>0$, and scale the points as follows
$$ \vec{w} = \frac1R \vec{z}. $$
One can construct a set of points $\vec{w}^0$ with the following properties:
\begin{itemize}
\item There are at most $p R^2$ points inside $D$ (the open unit disk).
\item Point to point separation: $|w_j^0 - w_k^0| \ge C t$ for all $j \ne k$.
\item Separation from boundary: $\bigcup_{w_j^0 \in D} D(w_j^0, t) \subset D$.
\item The points approximate the minimizing measure
\begin{equation}\label{min-meas-approx}
\left|U_{\mu_{\vec{w}^0}^t} (z) - U_{\overline{\mu}_p^\alpha} (z)\right| \le \frac{C \log N}{N} ,\quad \forall z \in \C.
\end{equation}
\end{itemize}

\begin{remark}
One can construct such a `good' set of points directly, using the radial symmetry of the problem. Another possibility is to use Fekete points for (essentially) the measure $\overline{\mu}_p^\alpha$ (that is, weighted Fekete points with respect to the weight given by $U_{\overline{\mu}_p^\alpha}$). A small difficulty with the second approach is that $p$ can depend on $N$ (the number of points).
\end{remark}

Next we define a set of `good configurations' of points:
$$
\cG_N = \left\{ \vec{w} \in \C^N : |w_j - w_j^0| < \frac{t}{2} \quad \forall j \in \{1,\dots,N\} \right\}.
$$
By the separation of the points, and using \eqref{min-meas-approx}, we have for $\vec{w} \in \cG_N$:
$$
\frac{1}{N^2} \sum_{j\ne k} \log|w_j - w_k| = \Sigma(\overline{\mu}_p^\alpha) + O\left(\frac{\log N}{N}\right),
$$
and
$$
\frac1N \sum_{j=1}^N |w_j|^2 = \int_\C |w|^2 \, \d \overline{\mu}_p^\alpha (w) + O\left(\frac{\log N}{N}\right).
$$
Also notice that the volume (in $\C^N$) of $\cG_N$ is at least $\exp(-C N \log N)$. The lower bound is obtained by integrating the (scaled) joint density over the set $\cG_N$, using the above estimates, and Claim \ref{func-value-at-mu-p}.

\subsection{Minimizing measures}\label{sec_min_meas}

The following lemma gives a characterization of the (constrained) minimizers of the functional $I_\alpha$ (cf. \cite[Lemma 10]{GN}).

\begin{lemma}\label{char_minimiz_Ginib}
Let $\cC \subset \probmeas$ be a closed and convex set of probability measures. The probability measure $\mu_0 \in \cC$ is the (unique) minimizer of $I_\alpha$ over $\cC$ if, and only if, for all $\mu \in \cC$,
$$ \int \frac{|z|^2}{2 \alpha} \, \d \mu_0 (z) - \Sigma(\mu_0) \le \int \frac{|z|^2}{2\alpha} \, \d \mu(z) - \int U_{\mu_0}(z) \, \d \mu(z). $$
\end{lemma}
\begin{proof}
Let $\mu \in \cC$ be a probability measure such that $\mu \ne \mu_0$. Without loss of generality we assume $\mu$ has compact support and finite logarithmic energy. It is known (\cite[Lemma I.1.8]{SaT}) that $- \Sigma(\mu - \mu_0) > 0$.

For $t \in (0,1)$ consider the measure $\mu_t = (1 - t) \mu_0 + t \mu \in \cC$, and expand $I_\alpha(\mu_t)$ to get:
\begin{align*}
I_\alpha(\mu_t) & = I_\alpha(\mu_0) + 2t \left( \int \frac{|z|^2}{2\alpha} \, \d \mu(z) - \int U_{\mu_0}(z) \, \d \mu(z) - \int \frac{|z|^2}{2 \alpha} \, \d \mu_0 (z) + \Sigma(\mu_0) \right)\\
& + t^2 \left[-\Sigma(\mu - \mu_0)\right].
\end{align*}
Therefore, if the linear term in $t$ is non-negative, then $I_\alpha(\mu_t) > I_\alpha(\mu_0)$, and this implies (from the convexity of $I_\alpha$) that $I_\alpha(\mu) > I_\alpha(\mu_0)$. The other direction is clear.
\end{proof}

We now wish to find the minimizing measure of the functional $I_\a$ over the set
$$\cF_p = \cF_p(\a) = \left\{ \mu \in \probmeas : \mu(D) \le \frac{p}{\alpha} \right\},$$
which is a closed and convex subset of $\probmeas$ (recall that $D$ denotes the \emph{open} unit disk).

We argue that the following probability measure is the minimizer
$$ \overline{\mu}_p^\alpha = \frac1\alpha \left(\indf_{\{|z| \le \sqrt{p}\}}  + \indf_{\{1 \le |z| \le \sqrt{\alpha}\}} \right) \frac{m}{\pi} + \frac{1-p}{\alpha} m_\T. $$
A simple calculation gives the values of the logarithmic potential on the support
$$ U_{\overline{\mu}_p^\alpha} (z) = \frac{|z|^2}{2 \alpha} + \frac{\log \alpha - 1}{2} + 
\begin{cases}
c_{1} & ,\,\left|z\right|\le\sqrt{p};\\
0 & ,\,1\le\left|z\right|\le\sqrt{\alpha},
\end{cases}
$$
where $c_1 = \frac{p(\log p - 1) + 1}{2 \alpha} > 0$.
 It is also not difficult to see that
$$ U_{\overline{\mu}_p^\alpha} (z) < \frac{|z|^2}{2 \alpha} + \frac{\log \alpha - 1}{2} + 
\begin{cases}
c_{1} & ,\, \sqrt{p} < \left|z\right| < 1;\\
0 & ,\, \sqrt{\alpha} < \left|z\right|.
\end{cases}
$$

The properties above, give
$$ \int \frac{|z|^2}{2 \alpha} \, \d \overline{\mu}_p^\alpha (z) - \Sigma(\overline{\mu}_p^\alpha) = 
\int \left[\frac{|z|^2}{2 \alpha} - U_{\overline{\mu}_p^\alpha}(z) \right] \, \d \overline{\mu}_p^\alpha (z) = 
-\frac{\log \a -1}{2} - c_1 \frac{p}{\a}.
$$

On the other hand, for any $\mu \in \probmeas$, such that $\mu(D) \le \frac{p}{\a}$, we have
$$ \int \left[\frac{|z|^2}{2 \alpha} - U_{\overline{\mu}_p^\alpha}(z) \right] \, \d \mu (z) \ge 
-\frac{\log \a -1}{2} - \int_D c_1 \, \d \mu \ge -\frac{\log \a -1}{2} - c_1 \frac{p}{\a}.
$$

Finally, it is straightforward to evaluate the functional $I_\alpha$ for $\overline{\mu}_p^\alpha$ (cf. \cite{Sh1}).

\begin{claim}\label{func-value-at-mu-p}
We have for $p\in [0,1)$
$$
I_\alpha(\overline{\mu}_p^\alpha) = I_\alpha(\overline{\mu}^\alpha_{\mathrm{eq}}) + \frac{1}{4\alpha^2} \left|2p^2 \log p - (p-1)(3p-1)\right|.
$$
\end{claim}

\begin{remark}
The above result also holds in the range $p>1$.
\end{remark}



\section{Analysis of the GEF zeros process}\label{cond-dist-GEF}
In order to prove Theorem \ref{conv-cond-meas-hole-GEF} and Theorem \ref{num-zeros-in-annulus-GEF} we need to understand the behavior of linear statistics of the GEF zeros, conditioned on the hole event in $\{ |z| < R \}$. We recall, that linear statistics are the random variables
$$ n_F(\varphi ; R) = \sum_{z :\, F(z) = 0} \varphi\left(\frac{z}{R}\right), $$
where $\varphi$ is a smooth (say $C^2$) test function with compact support, and we sum over all the zeros of the GEF (to prove Theorem \ref{num-zeros-in-annulus-GEF} the test function $\varphi$ has to depend on $R$, but here we ignore this technicality).

The analysis presented in the previous section requires some modifications in order to obtain the conditional intensity for the GEF zeros process. The main idea is to approximate the (scaled) GEF with the Weyl polynomials
\[
P_{\a,R}(z) = \sum_{k=0}^N \frac{\xi_k}{\sqrt{k!}} (Rz)^k,
\]
where $\a>1$ is a large parameter (eventually depending on $R$), and $N = N(\a, R) = {\lfloor \a R^2 \rfloor}$ is the degree of $P_{\a,R}$.

Roughly speaking, with this choice of the parameters the scaled GEF $\f(Rz)$ (defined in \eqref{GEF}) and the polynomial $P_{\a,R}$ have a very similar bevaviour inside a disk of radius $\sqrt{\beta}$, as long as $\beta \ll \a$. Therefore, by taking $\a$ large we can obtain an understanding of the conditional intensity of the Gaussian zeros (under conditioning by $\H_R$), by analyzing the same problem for the polynomials.

\begin{remark}
It turns out that in order to carry out this approximation scheme, we need to let $\a \to \infty$ logarithmically with $R$. A drawback of this is that we cannot use `off-the-shelf' large deviation principles for empirical measures of random polynomial zeros such as \cite{ZZ}. 
\end{remark}

\subsection{Joint probability density and the limiting functional}
The joint density of the zeros of $P_{\a,R}$ with respect to Lebesgue measure on $\C^N$ is given by (cf. \eqref{weyl-joint-dens})
\begin{equation}
\label{density}
\rho_R(z_1, \dots, z_N) = (C_{N,R,\alpha})^{-1} {|\D(z_1, \dots, z_N)|^2}{\l(\frac{R^2}{\pi}\int_\C |Q_N(w)|^2 e^{-R^2|w|^2} \, \d m(w) \r)^{-(N+1)}},                                                                                                                                                                                                             \end{equation}
where $Q_N$ is the polynomial
\[ Q_N(w) = (w-z_1)\dots (w-z_N), \]
and $C_{N,R,\alpha}$ is a normalizing constant. As was briefly outlined in \ref{weyl-polynomials}, at the exponential scale one can approximate this density with the limiting functional
\[
I^Z_\a(\mu)= 2 \sup_{z \in \C} \l( U_\mu(z) - \frac1{2 \a}|z|^2 \r) - \S(\mu) - C_\a,
\]
where we used $\frac{N}{R^2} \approxeq \a$.

\begin{remark}
The global minimizer of the functional $I^Z_\a$ is the uniform probability on the disk $\{|z| \le \sqrt{\a}\}$, which we denoted by $\overline{\mu}^\alpha_{\mathrm{eq}}$.
\end{remark}

\subsubsection{Deviation inequality}
The analysis is done in a similar way to the case of the $\beta$-Ginibre ensembles (Sections \ref{dev-ineq-ginibre} and \ref{approx-joint-intens}).

Recall that $N_F(R)$ is the number of zeros of the GEF inside the disk $\{ |z| < R \}$. To fix ideas, consider the following event
$$ F^Z_p(R) = \left\{ N_F(R) \le p R^2 \right\}, \quad p\in[0,1),$$
and introduce the following sets of measures,
$$ \cF_p = \left\{ \mu \in \probmeas : \mu(D) \le \frac{p}{\alpha}  \right\}, \quad p\in [0,1) $$
and
$$ \cL^Z_{\varphi, \lambda} = \left\{ \mu \in \probmeas : \left| \int_\C \varphi(w) \, \d \mu (w) - \int_\C \varphi(w) \, \d \overline{\nu}_p^\alpha (w) \right| \ge \lambda \right\}, $$
where $D$ is the open unit disk, and $\overline{\nu}_p^\alpha$ is the restriction of the limiting measure $\mu^Z_p$ (introduced in Section \ref{defic-over}) to the disk $\{|z| \le \sqrt{\a}\}$, and normalized to be a probability measure.

In a similar fashion to \eqref{prob-lin-stats-upper-bnd}, one can show
\begin{align*}
\P \left[F^Z_p(R) \cap \left\{\left| n_F(\varphi; R) - R^2 \int_\C \varphi(w) \, \d \mu^Z_p(w) \right| \ge \lambda \right\} \right]\quad\quad\quad\quad\quad\quad\quad\quad\quad\quad \nonumber \\
\le \exp\left(- N^2 \inf_{\mu\in \cF_p \cap \cL^Z_{\varphi,\widetilde{\lambda}}} \left[I^Z_{\alpha} (\mu) - I^Z_\alpha(\overline{\mu}^\alpha_{\mathrm{eq}}) \right] + O(N \log N) \right), 
\end{align*}
where $\widetilde{\lambda} = \frac{\lambda}{2 N}$. Then, combining Claim 11 from \cite{GN} (corresponding to Claim \ref{convexity-of-functional} for the Ginibre ensemble), together with a lower bound estimate for the probability of $F^Z_p(R)$, we obtain the deviation inequality
$$ \P_{F^Z_p(R)} \left[\left|n_F(\varphi ; R) - R^2 \int_\C \varphi(w) \, \d \mu^Z_p (w) \right| \ge \lambda \right] \le 
\exp\left(-\frac{C}{\fD (\varphi)} \lambda^2 + C_\varphi R^2 \log^2 R\right).
$$

\begin{remark}
The actual proof is technically more involved, in large part because the choice of the value of $N$ has to be random.
\end{remark}

\subsubsection{Lower bound for $\P\left[F^Z_p(R)\right]$}\label{low-bnd-for-GEF}
Because of the circular symmetry of the problem, one can use analytic techniques to obtain the lower bound for the probability of the event $F^Z_p(R)$, which are not available in the case of the $\beta$-Ginibre ensembles.

The main idea is to use Rouch{\'e}'s theorem. More precisely, recalling that the GEF is given by the Gaussian Taylor series
\[ \\\f(z)=\sum_{k=0}^{\infty} \frac{\xi_k}{\sqrt{k !}} z^k,
\]
we explicitly construct an event where the term $\left|\xi_{k_{0}}\frac{z^{k_{0}}}{\sqrt{k_{0}!}}\right|, \,\, k_0 = {\lfloor p R^2 \rfloor}$  dominates the sum over all the other terms (on the circle $\left\{ \left|z\right|=R\right\}$). This simple but effective method originally appeared in the paper $\cite{STs-1}$, and was later used in many other problems of this type.

\subsubsection{The minimizing measures}
In order to find the limiting conditional measures for the GEF zeros process, one has to identify the (probability) measure $\overline{\nu}_p^\alpha$ which minimizes the functional $I^Z_\a$ over the set $\cF_p$. The interested reader can find the details in \cite[Section 5]{GN}.

\section{Conditional intensities in 1D}
The conditional intensity around a hole has been studied in 1D (where it is usually called a `gap'), in the context of the \emph{GUE} (Gaussian Unitary Ensemble) point process in \cite{MaS}. In this section we briefly describe the approach  in \cite{MaS}, and compare and contrast the results therein with the situation we already discussed in 2D.

In 1D, under the natural scaling (necessary for obtaining an LDP from the GUE), the ``droplet'' (that is, the minimizer of the LDP rate functional) assumes the form of the famous \textsl{semicircle distribution}. With the normalization of \cite{MaS}, this density is given by
$$ f(x) = \frac1\pi \sqrt{2 - x^2}, \quad |x| \le \sqrt{2}. $$
Under this scaling, the ``hole'' assumes the form of an interval $(\zeta_1,\zeta_2)$. In the important case where $(\zeta_1,\zeta_2)$ is an interval symmetric about the origin, denoted $(-w,w)$, the (scaled) conditional intensity has the form 
\[ f_w(x) = \frac{1}{\pi} \sqrt{  \frac{ L^2 - x^2  }{ x^2 - w^2} } |x|, \quad x \in [-L,-w] \cup [w,L],
\]
where $L=\sqrt{w^2+2}$. In particular, we note that there is no singular component, in contrast with both of the models we considered in 2D. Figure \ref{one-dimen-cond-intens} depicts the density for $w = 1$.

\begin{figure}[ht]
\label{1D}
    \centering
    \begin{minipage}{0.55\textwidth}
        \centering
        \includegraphics[width=1.0\textwidth]{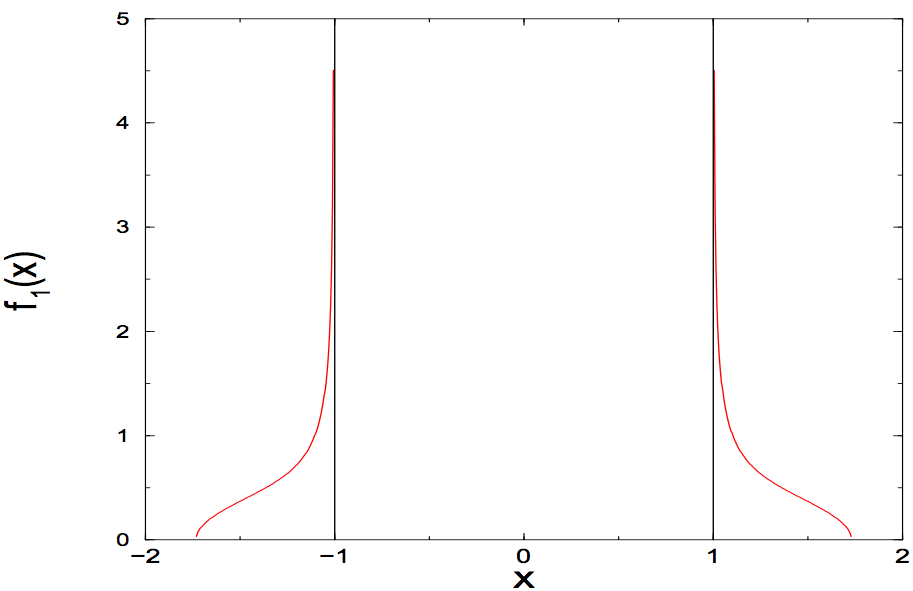} 
        \caption{Conditional intensity profile in 1D}\label{one-dimen-cond-intens}
    \end{minipage}\hfill
\end{figure}

In \cite{MaS}, the authors approach this problem by obtaining a singular integral equation for $f_w$, which is deduced essentially from a variational perturbation of the LDP rate functional around the minimizing measure. The authors then illustrate two approaches to solving this singular integral equation - one of them being a Riemann-Hilbert approach and the other being via an application of Tricomi's theorem (\cite{Tr}). 

In  \cite{MaS}, the authors also study the particle distribution for \textsl{atypical indices} for the GUE. The \textsl{index} $N_+$ of a configuration of $N$ particles on $\R$ is the number of particles on $\R_+$. By symmetry, the typical value of $N_+/N=1/2$. Using similar variational techniques (as discussed above) on the LDP rate functional for the GUE ensemble, in \cite{MaS} the authors obtain the asymptotics of the probability of an atypical index $N_+/N = c \ne 1/2$, as well as the typical particle profile given such an atypical index. This can be compared with \cite{ArSZ}, where a similar problem has been studied for the Ginibre ensemble using free boundary techniques. 

\section{Simulation of the hole event and numerical aspects}
Numerical methods to effectively simulate the distribution of zeros (or eigenvalues), conditioned on a large hole, is a challenging problem, because of the rarity of the hole event and the strong correlation among the particles. E.g., for the Ginibre ensemble, the asymptotics of the hole probabilities can be understood via the statistical independence of their absolute values, but this approach is not useful for simulations, because it carries no information about the correlations between the particles.

In the present paper, Figure \ref{GEF_zeros_hole} is obtained by simulating the GEF (and then the zeros thereof) under the conditions (on the coefficients) that produce the tight lower bound for the hole probability (as alluded to in Section \ref{low-bnd-for-GEF}). Figure \ref{Ginib_hole} is obtained by manually moving the eigenvalues of a Ginibre matrix (that are inside the disk)  to the disk's  boundary. In \cite{MaS}, a modified Metropolis Hastings algorithm was studied for simulating such conditional distributions (conditioned on hole, overcrowding or deficiency events), for the GUE process in 1D.

A 2D analogue of such an algorithm, for the Ginibre ensemble, would consist in the following:
We start with a ``legitimate'' particle configuration, namely one that satisfies the constraint of having a hole. E.g., a reasonable initial configuration would be equi-spaced points on the boundary of the hole.
Given a ``legitimate'' configuration $(\la_1,\dots,\la_N)$, we generate a new one $(\la_1',\dots,\la_N')$ by perturbing a (randomly picked) particle by a small Gaussian noise, conditioned to avoid the ``hole''. Ideally, we then replace the current configuration with the new one with probability
$$\min \l( \frac{f_N(\la_1',\dots,\la_N')}{f_N(\la_1,\dots,\la_N)} ,1 \r),$$
where $f_N$ is the probability density function of the Ginibre ensemble of size $N$. However, the generation of new configurations that avoid the hole  introduces an inherent asymmetry, which have to be taken into account in the acceptance probability for this approach. More precisely, one has to add the ratio of the probability to move from the new configuration to the old one, over the probability to move from the old configuration to the new one (which are not the same).

Figure \ref{Ginib-simul-1} presents the result of $10000$ iterations of the above algorithm, in the case of a circular hole. The initial configuration consists of $2000$ points uniformly distributed in the annulus outside of the excluded disk, and inside the support of the equilibrium measure (indicated by the outer circle). It is more difficult to implement this algorithm efficiently for the GEF zero process, mainly because the finite particle density $f_N$  is considerably more complicated for the zeros, involving  interactions of all orders up to $N$.

\begin{figure}[ht]
\label{1D}
    \centering
    \begin{minipage}{0.55\textwidth}
        \centering
        \includegraphics[width=1.0\textwidth]{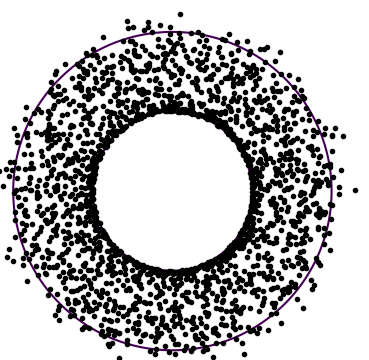} 
        \caption{Simulation of a circular hole by a modified Metropolis Hastings algorithm}\label{Ginib-simul-1}
    \end{minipage}\hfill
\end{figure}

A different but related question is to numerically solve constrained optimization problems  on the space of probability measures (on a Euclidean space), like the ones arising in the hole problem for the Ginibre ensemble and the GEF zeros. In the setting of the Ginibre ensemble, the goal is to minimize a weighted logarithmic energy \[ I(\mu) = \int V(x) \, \d \mu(x) - \S(\mu) \] of a probability measure (or, more generally, a finite Borel measure) $\mu$, subject to constraints on its support. This can be directly related to LDP rate functionals -  for example, the LDP rate functional for the Ginibre process is a logarithmic energy with a quadratic weight $V$.

To our knowledge, very little is known about this problem in dimensions greater than one. In 1D, a similar numerical problem has been addressed by \cite{CDV} in the weighted case (using an approach involving iterated Balayage), and by \cite{HV}, \cite{CD} in the unweighted case.  Even in the 1D situation, there are various assumptions on the Riesz measure corresponding to the weight $V$, which would be of interest to relax.

It remains a non-trivial and highly interesting question  to devise efficient numerical techniques to simulate the particle configurations for the hole (and, in the same vein, for overcrowding and deficiency) events. In the case of the hole event for the Ginibre ensemble, one can use weighted Leja points (see \cite[Chapter V]{SaT}) to approximate the most likely eigenvalue configurations (given by the weighted Fekete points, mentioned in Section \ref{genhole}). Finding a similar method for the GEF zeros process seems to be an interesting problem.

\section{Acknowledgements}
We thank the authors of the paper \cite{MaS} for allowing us to use the picture in Figure \ref{one-dimen-cond-intens}. We thank Diego Ayala for allowing us to use the picture in Figure \ref{Ginib-simul-1}. We thank the anonymous referee for numerous helpful suggestions. The work of S.G. was supported in part by the ARO grant W911NF-14-1-0094 and the NSF grant DMS-1148711.

\end{document}